\documentclass[11pt]{amsart}
\usepackage{graphicx}
\usepackage{amssymb}

\newtheorem{remark}{Remark}

\newtheorem{proposition}{Proposition}

 \title{Topics in Hyperplane Arrangements, Polytopes, and Box Splines--Errata}
\author{C. De Concini,\quad
C. Procesi}

\begin{document}\begin{abstract}
We have received an e-mail from Bryan Gillespie pointing out that a proposition, that is Proposition
8.5  of our book  \cite{depr}, is incorrect as stated. The given formula (8.5) is valid only in the generic case
that is assuming that for any point of the arrangement $p$, $X_p$ is formed by a basis. The correct
proposition is slightly weaker, in general one must replace Formula 8.5 of the book with the next
Formula (3). Accordingly one has to change Proposition 9.2 in the obvious way.
The remaining parts of the book are not affected but one should remove the first line of 11.3.3 which
quotes the incorrect formula. Here we discuss the correct proposition, replacing Proposition 8.5.
 \end{abstract}
\maketitle

Let us first develop a simple identity.  Take vectors  $b_i,\ i=0,\dots,k$. Assume that  $b_0=\sum_{i=1}^k\alpha_ib_i.$ Choose numbers $\nu_i,\ i=0,\dots,k,$ and set
\begin{equation}\label{nu} \nu:=\nu_{ 0}-\sum_{i=1}^k\alpha_i\nu_{i}.\end{equation}  If
$\nu \neq 0,$ we write
\begin{equation}\label{separat}\frac{1}{ \prod_{i=0}^k(b_i+\nu_{i})}=\nu^{-1}  \frac{b_0+\nu_{0}- \sum_{i=1}^k\alpha_i(b_i+\nu_{i})}{ \prod_{i=0}^k(b_i+\nu_{i})}.\end{equation}
When we  develop the right-hand side, we obtain a  sum of $k+1$
terms in each of which one of the elements $b_i+\nu_{i}$ has
disappeared. Let us remark that if $\alpha_i\neq 0$ the span of $b_0\ldots ,\check b_i,\ldots b_k$ equals the span of $b_1,\ldots ,b_k$.

\smallskip

Let us recall some notation, we let $X = \{a_1,\dots,a_m\}$ be a list of vectors spanning a real (or complex)
vector space $V$ and $\underline\mu:= \{\mu_1,\dots,\mu_m\} $ a list of real (resp. complex) parameters. These data define
a hyperplane arrangement in  $V^*$ given by the linear equations $a_i +\mu_i = 0$, the various intersections
of these hyperplanes form the subspaces of the arrangement. In particular we have {\em the points
of the arrangement} for which we use the notation $P(X,\underline \mu)$   of Section 2.1.1. Given $p\in P(X,\underline \mu)$  we
set $X_p$ for the sublist of $a \in X$ such that  $a+\mu_a$  vanishes at $p$. Denote by $\mathcal L_{X_p}$ the family of subsets
of $X_p$ spanning $V$. Notice that if  $\ell\in\mathcal L_{X_p}$ the linear polynomials $a+\mu_a$ with $a\in \ell$ have $p$ as unique common zero.\begin{proposition}[Replaces 8.5 of \cite{depr}] \label{separano} Assume that $X$ spans $V$.
Then:
\begin{equation}\label{separa}\prod_{a\in X}\frac{1}{   a  +\mu_a}=\sum_{p\in  P(X,\underline \mu)} \sum_{\ell\in \mathcal L(X_p)} c_\ell\prod_{a\in \ell}  \frac{1}{   a  +\mu_a}=\sum_{p\in  P(X,\underline \mu)} \sum_{\ell\in \mathcal L(X_p)}c_\ell\prod_{a\in \ell}  \frac{1}{   a  -\langle a\,|\, p\rangle} \end{equation}
with $c_\ell\in \mathbb C$.

For any $p\in P(X,\underline \mu)$,
$$c_{X_p}=\prod_{a\in X\setminus X_p}\frac{1}{   \langle a\,|\, p\rangle  +\mu_a}.$$
\end{proposition}
\begin{proof}  This follows by induction applying the previous algorithm of separation of denominators.

Precisely, if $X$ is a basis, there is a unique point of the
arrangement and there is nothing to prove. Otherwise,  we can write $X=(Y,z)$
where $Y$ still spans $V$. By induction
$$\prod_{a\in X}\frac{1}{   a  +\mu_a}=\frac{1}{ z+\mu_z}\prod_{a\in Y}\frac{1}{   a  +\mu_a}=\sum_{p\in  P(Y,\underline \mu)} \sum_{\ell\in \mathcal L(Y_p)}c_\ell\frac{1}{ z+\mu_z}\prod_{a\in \ell}\frac{1}{   a +\mu_a}. $$
We need to analyze each product  \begin{equation}\label{illpp}
\frac{1}{ z+\mu_z}\prod_{a\in
\ell}\frac{1}{   a  +\mu_a}.
\end{equation}  If $\langle z\,|\,p\rangle+\mu_z=0,$ then
$p\in P(X,\underline \mu), \ \ell\cup \{z\}\in \mathcal L(X_p)$ and we are done.
Otherwise,since $\ell$ spans $V$, write $z=\sum_{a\in \ell}d_aa$ and apply the
previous  algorithm to the   list $\{z\}\cup \ell$ and the corresponding  numbers $\mu_z,\ \mu_a$. As we have remarked
the product \eqref{illpp} develops as a linear combination of
products  of the form
$$\prod_{a\in \ell'}\frac{1}{   a  +\mu_a}$$
$\ell'$ a proper subsequence of $\{z\}\cup\ell$ and hence  of $X$ whose elements span $V$. So we can proceed by induction.

It remains to compute $c_{X_p}$.  For a given $p\in P(X,\underline
\mu)$,
$$\prod_{a\in X\setminus X_p}\frac{1}{   a  +\mu_a}=c_{X_p}+\sum_{q\in  P(X,\underline \mu)} \sum_{\ell\in \mathcal L(X_q), \ell\neq X_p}
c_\ell\frac{\prod_{a\in X_p}  (a  +\mu_a)}{  \prod_{a\in \ell} (a  +\mu_a)}.$$

Hence, evaluating both sides at $p$ yields
$$c_{X_p}=\prod_{a\in X\setminus X_p}\frac{1}{   \langle a\,|\, p\rangle  +\mu_a}.$$ \end{proof}
Given any list $X$ spanning $V$ it is easily seen that for generic values of the parameters each set
$X_p$  is a basis of $V$ extracted from $X$ and each basis of  $V$ extracted from $X$ gives rise to a point of
the arrangement. We say then that $X,\underline\mu$ {\em are generic.}
\begin{remark}\label{rem}
 In case the data $X,\underline\mu$  are generic the set $\mathcal L(X_p)$ reduces to the single element $ X_p$ and
Formula (3) gives back Formula 8.5.

\end{remark}One can also reformulate the formula  as

\begin{equation}\label{separa1}\prod_{a\in X}\frac{1}{   a  +\mu_a}=\sum_{p\in  P(X,\underline \mu)}  C_p\prod_{ a  +\mu_a \in X_p}  \frac{1}{   a  +\mu_a}  \end{equation}
with $C_p $  no more a number but a polynomial. 

In fact we can replace each term
$$ \prod_{a\in \ell}\frac{1}{   a +\mu_a}=\prod_{a\in X_p\setminus \ell}(   a +\mu_a) \prod_{a\in X_p}\frac{1}{   a +\mu_a}  $$ and then collect the terms so that 
\begin{equation}\label{separa2}C_p=\sum_{\ell\in \mathcal L(X_p)} c_\ell \prod_{a\in X_p\setminus \ell}(   a +\mu_a). \end{equation}

As the reader will notice, if $\ell\neq X_p$,  there is no Formula for the coefficients  $c_\ell$ this is due to the fact that these coefficients are  not uniquely determined, that is the expansion of Formula  \eqref{separa} is in general not unique, which is clear from Formula \eqref{separa2}.

 \bibliographystyle{plain}

\bibliography{bibliografia}

\end{document}